\documentclass[11pt,a4paper]{article}
\usepackage{amsmath}
\usepackage{amssymb}
\usepackage{amsthm}
\usepackage{amsfonts}
\usepackage{enumerate}
\usepackage{verbatim}
\usepackage{hyperref}
\usepackage{breakurl,url}
\usepackage{bm}
\usepackage{mathtools} %for := \coloneqq, for parentheses
\usepackage{pstricks}
\usepackage{pst-node}
\usepackage{pst-plot}
\usepackage{algorithm}
\usepackage{algorithmic}
\usepackage{pst-poly}
\usepackage{pgfpages} %MUST BE FOR PSTpictures
\usepackage{booktabs}
\usepackage{caption,subcaption} %subfigures
\usepackage[margin=1.3in]{geometry}
\newcommand{\mna}[1]{{\mathcal{#1}}}

\newcommand{\mmid}[0]{;\,}		%
\newcommand{\st}[0]{{\ \ \mbox{subject to}\ \ }}
\newcommand{\stl}[0]{{\;\mbox{subject to}\ \ }}

\def\tluste#1{\bm{#1}}
\newcommand{\omace}[1]{\mbox{$\overline{#1}$}}	%upper limit of int matrix
\newcommand{\umace}[1]{\mbox{$\underline{#1}$}}  %lower limit of int matrix
\newcommand{\imace}[1]{\tluste{#1}} 		%interval matrix
\def\Mid#1{{#1}_c}		%center of interval
\def\Rad#1{{#1}_\Delta}		%radius of interval
\DeclareMathOperator{\diag}{diag}	%diagonal matrix
\DeclareMathOperator{\sgn}{sgn}	%sign

\DeclarePairedDelimiter\parentheses{\lparen}{\rparen}	%for () after an operator
\DeclarePairedDelimiter\braces{\lbrace}{\rbrace}	%for {} after an operator

\newcommand{\ovr}[1]{\mbox{$\overline{#1}$}}	%upper limit int vector
\newcommand{\uvr}[1]{\mbox{$\underline{#1}$}}	%lower limit int vector
\newcommand{\ivr}[1]{\tluste{#1}} 		%int vector
\newcommand{\onum}[1]{\overline{{#1}}} 	%upper limit interval
\newcommand{\unum}[1]{\underline{{#1}}} 	%lower limit interval
\newcommand{\inum}[1]{\tluste{#1}} 		%interval

\newcommand{\R}[0]{{\mathbb{R}}}

\newcommand{\IR}[0]{{\mathbb{IR}}}	%set of real interval
\newcommand{\seznam}[1]{{\{1, \ldots, {#1}\}}}

\newcommand{\nic}[1]{{}}

\newtheorem{theorem}{Theorem}

\newtheorem{proposition}{Proposition}

\newtheorem{corollary}{Corollary}

\theoremstyle{definition}

\newtheorem{example}{Example}
\begin{document}

\title{Range of optimal values in absolute value linear programming with interval data}

\author{
  Milan Hlad\'{i}k\footnote{
Charles University, Faculty  of  Mathematics  and  Physics,
Department of Applied Mathematics,
Malostransk\'e n\'am.~25, 11800, Prague, Czech Republic,
e-mail: \texttt{hladik@kam.mff.cuni.cz}}
}

\date{\today}
\maketitle

\begin{abstract}
Absolute value linear programming problems is quite a new area of optimization problems, involving linear functions and absolute values in the description of the model. In this paper, we consider interval uncertainty of the input coefficients. Our goal is to determine the best and the worst case optimal values. For the former, we derive an explicit formula, reducing the problem to a certain optimization problem. However, the latter is more complicated, and we propose a lower and upper bound approaches to estimate the value. We also investigate the basis stability, in which situation the best case optimal value is efficiently computable. The worst case optimal value then also admits a simple characterization; however, the computational complexity remains open.
\end{abstract}


\bigskip
%\paragraph{MSC Codes.}   90C30, 90C26, 90C33, 65G40,  15A06

%%%%%%%%%%%%%%%%%%%%%%%%%%%%%%%%%%%%%%%%%%%%%%%%%%%%%%%%%%%%%%%
% INTRODUCTION
%%%%%%%%%%%%%%%%%%%%%%%%%%%%%%%%%%%%%%%%%%%%%%%%%%%%%%%%%%%%%%%
\section{Introduction}

%%%%%
\paragraph{Absolute value linear programming.}
Absolute value linear programming (AVLP) problems \cite{Man2007} are mathematical programming problems involving linear functions and absolute values. The study of this class of problems was initiated by Mangasarian~\cite{Man2007} and later addressed by several papers \cite{ChenHan2025u,HlaHar2023au,Man2015b,YamFuk2014}. The canonical form of these problems introduced in \cite{HlaHar2023au} has the form of
\begin{align}\label{avlp}
f(A,b,c,D)=\max\ c^Tx \st Ax-D|x|\leq b,
\end{align}
where $c\in\R^n$, $b\in\R^m$ and $A,D\in\R^{m\times n}$. It is assumed that $D$ is nonnegative, since otherwise the negative coefficients can be equivalently reformulated by linear constraints. The presence of the absolute values makes the problem computationally difficult. Indeed, it is NP-hard; for instance, integer linear programs can be reformulated in this form.

In this paper, we study the variations of the optimal value $f(A,b,c,D)$ subject to variations of the input entries in the interval domains. To formulate our problem precisely, we have to introduce some notation, in particular from interval computation.

%%%%%
\paragraph{Notation.}
We use $I_n$ for the identity matrix of size $n\times n$ and $\diag(s)$ for the diagonal matrix with entries given by vector $s\in\R^n$. 
The $i$th canonical unit vector is denoted by~$e_i$ and the all-ones vector of convenient length by $e=(1,\dots,1)^T$. 
Further, $A_B$ denotes the submatrix of $A\in\R^{m\times n}$ determined by the rows of $A$ indexed by~$B$. Similarly, $v_B$ denotes the subvector of $v\in\R^n$ determined by the index set~$B$. 
We use $A_B^T$ and $A^{-T}$ with the meaning $(A_B)^T$ and $(A^{-1})^T$, respectively. 
For a square matrix $A$, the symbol $\rho(A)$ stands for its spectral radius, $\sigma_{\min}(A)$ for its minimal and $\sigma_{\max}(A)$ for its maximal singular value.  
The sign of a real r is defined as $1$ if $r\geq0$ and $-1$ otherwise. 
Inequalities, the sign and absolute value functions are applied entrywise for vectors and matrices. 

%%%%%
\paragraph{Interval computation.}
An interval matrix is defined as the set
\begin{align*}
\imace{A}=[\umace{A},\omace{A}]=\{A\in\R^{m\times n}\mmid \umace{A}\leq\omace{A}\},
\end{align*}
where $ \umace{A},\omace{A}\in\R^{m\times n}$, $\umace{A}\leq\omace{A}$, are given matrices. The interval matrix can equivalently be defined by its midpoint and radius defined, respectively, as
\begin{align*}
\Mid{A}=\frac{1}{2}(\umace{A}+\omace{A}),\quad
\Rad{A}=\frac{1}{2}(\omace{A}-\umace{A}).
\end{align*}
The set of all interval matrices of size ${m\times n}$ is denoted by $\IR^{m\times n}$. 
Interval vectors are regarded as one-column interval matrices. 
Following Rohn~\cite{Roh2006a}, we make use of the following notation: for $\imace{A}\in\IR^{m\times n}$, $r\in[-1,1]^m$ and $s\in[-1,1]^n$, we define
\begin{align*}
A_{r,s} \coloneqq \Mid{A}-\diag(r)\Rad{A}\diag(s) \in\imace{A},
\end{align*}
and for $\ivr{c}\in\IR^n$ and $s\in[-1,1]^n$, we define
\begin{align*}
c_{s} \coloneqq \Mid{c}+\diag(s)\Rad{c} \in\ivr{c}.
\end{align*}
Interval arithmetic is defined in many textbooks; see e.g.\ \cite{HanWal2004,MooKea2009}.
An interval matrix $\imace{A}$ is called regular if every $A\in\imace{A}$ is nonsingular. Testing regularity is a co-NP-hard problem~\cite{PolRoh1993}, but some sufficient conditions exist, such as the condition by Beeck~\cite{Beeck1975}
\begin{align}\label{sufCondRegBeeck}
\rho(|\Mid{A}^{-1}|\Rad{A}) < 1,
\end{align}
or condition by Rex \& Rohn~\cite{RexRoh1998}
\begin{align}\label{sufCondRegRexRohn}
\sigma_{\max}(\Rad{A}) < \sigma_{\min}(\Mid{A}).
\end{align}

%%%%%
\paragraph{Preliminaries.}
AVLP problems can be seen as an extension of two problems, and both of them will be addressed in this paper. First, they extend the class of linear programming (LP) problems
\begin{align}\label{lp}
\max\ c^Tx \st A^*x\leq b.
\end{align}

Second, AVLP problems generalize the problem of solving generalized absolute value equations (GAVE)
\begin{align*}
\tilde{A}x + \tilde{B}|x| = \tilde{b},
\end{align*}
where $\tilde{A},\tilde{B}\in\R^{n\times n}$ and $\tilde{b}\in\R^n$. This problem has been intensively studied in the recent years \cite{KumDee2024,HlaMoo2024au,Mez2020,RohHoo2014,WuShe2021}. In particular, it was shown \cite{WuShe2021} that nonsingularity of $\tilde{A}+\tilde{B}D$ for every $D\in[-I_n,I_n]$ implies that the GAVE has a unique solution for each $\tilde{b}\in\R^n$. This nonsingularity condition is NP-hard to check; however, there are two convenient sufficient conditions
\begin{align*}
\rho(|\tilde{A}^{-1}\tilde{B}|)<1
\end{align*}
and 
\begin{align*}
\sigma_{\max}(\tilde{A}^{-1}\tilde{B})<1.
\end{align*}
%any of them implies the unique solvability of GAVE \cite{RohHoo2014,WuShe2021}. 
Moreover, if any of these conditions succeeds, then the GAVE is solvable in polynomial time; see \cite{ManMey2006,ZamHla2021a}.

%%%%%
\paragraph{Interval AVLP problem.}
Let interval vectors $\ivr{c}\in\IR^n$ and  $\ivr{b}\in\IR^m$ and interval matrices $\imace{A},\imace{D}\in\IR^{m\times n}$ be given. Then the \emph{interval AVLP problem} is introduced as 
\begin{align}\label{intAvlp}
\max\ \ivr{c}^Tx \st \imace{A}x-\imace{D}|x|\leq \ivr{b}.
\end{align}
However, its interpretation is that it is a family of AVLP problems \eqref{avlp} with $A\in\imace{A}$, $b\in\ivr{b}$, $c\in\ivr{c}$ and $D\in\imace{D}$. 

Throughout the paper, we assume that $\umace{D}\geq0$. Basically, this assumption is without loss of generality, but simplifies the overall analysis. The best case optimal value is achieved for $D\coloneqq\omace{D}$ and the worst case optimal value for $D\coloneqq\umace{D}$, so the problem can be reduced to a fixed real matrix~$D$. The absolute values associated with negative entries of $D$ can be linearized by standard techniques, so the difficult cases are the nonnegative entries. Assuming $\umace{D}\geq0$ ensures that both $\umace{D}$ and $\omace{D}$ are nonnegative.

%%%%%
\paragraph{The goal: Range of optimal values.} 
The aim of this paper is to compute the best and the worst case optimal values defined, respectively, as
\begin{align*}
\onum{f} &= \max\ f(A,b,c,D) \st A\in\imace{A},\,b\in\ivr{b},\,c\in\ivr{c},\,D\in\imace{D},\\
\unum{f} &= \min\ f(A,b,c,D) \st A\in\imace{A},\,b\in\ivr{b},\,c\in\ivr{c},\,D\in\imace{D}.
\end{align*}

This problem was thoroughly discussed for interval-valued LP problems \cite{ChinRam2000,GabMur2010,GarHla2019c,GarRad2023,GarRad2024,Hla2009b,Hla2025a,Mra1998,Roh2006b}. Certain extensions are also known general interval nonlinear programming problems \cite{HanWal2004,Hla2011b} or for special classes such as interval convex quadratic programming \cite{Hla2017c,Hla2024a,LiXia2015}, and geometric programming \cite{Hla2024a,Liu2008}. However, for the AVLP problems, it seems that the optimal value range has not been investigated so far.

%%%%%%%%%%%%%%%%%%%%%%%%%%%%%%%%%%%%%%%%%%%%%%%%%%%%%%%%%%%%%%%
\section{General results}

%%%%%
\paragraph{Best case optimal value.} 
For the best case optimal value $\onum{f}$ we have a complete characterization by means of a reduction to one real AVLP problem. 

\begin{proposition}\label{propOfAvlp}
We have
\begin{align}
\label{ofAvlp}
\onum{f} 
  &= \max\ \Mid{c}^Tx+\Rad{c}^T|x| \st \Mid{A}x-(\Rad{A}+\omace{D})|x| \leq \ovr{b} \\ 
\label{ofDecomp}
  &= \max_{s\in\{\pm1\}^n}\ \max\ c_s^Tx \st (A_{e,s}-\omace{D}\diag(s))x \leq \ovr{b},\ 
     \diag(s)x \geq 0.
\end{align}
\end{proposition}

\begin{proof}
Obviously, the best case realizations satisfy $b=\ovr{b}$ and $D=\omace{D}$. Thus, can write
\begin{align*}
\onum{f} 
 &= \max\ f(A,\ovr{b},c,\omace{D}) \st A\in\imace{A},\,c\in\ivr{c} \\
 &= \max\ \parentheses*{\max_{c\in\ivr{c}}\, c^Tx}
   \st Ax-\omace{D}|x| \leq \ovr{b},\ A\in\imace{A} \\
 &= \max\ \Mid{c}^Tx+\Rad{c}^T|x| \st \Mid{A}x-(\Rad{A}+\omace{D})|x| \leq \ovr{b}.
\end{align*}
Using the orthant-by-orthant decomposition of space $\R^n$, we can write $|x|=\diag(s)x$ for the orthant characterized by $\diag(s)x\geq0$. Therefore, within each orthant, the description of the problem becomes linear as in~\eqref{ofDecomp}.
\end{proof}

Formulation \eqref{ofAvlp} reduces the problem of computing the best case optimal value to one real-valued AVLP problem. In contrast, \eqref{ofDecomp} reduces the problem to $2^n$ LP problems. 
From \eqref{ofAvlp} we see that $\onum{f}$ is always attained as the optimal value of some realization, and the realization has the form of
\begin{align*}
A=A_{e,s},\ \ b=\ovr{b},\ \ c=c_s,\ \ D=\omace{D}
\end{align*}
for some $s\in\{\pm1\}^n$.

From the theory of interval linear programming \cite{ChinRam2000,Hla2025a,Mra1998,Roh1980}, formula \eqref{ofAvlp} can be interpreted as the best case value of a certain interval LP problem.

\begin{corollary}\label{corOnumIlpBest}
The value $\onum{f}$ is the same as the best case optimal value of the interval LP problem
\begin{align}\label{minCorOnumIlpBest}
\max\ \ivr{c}^Tx \st  \imace{A}^* x \leq \ivr{b},
\end{align}
where $\imace{A}^*=\big[\Mid{A}-\Rad{A}-\omace{D},\,\Mid{A}+\Rad{A}+\omace{D}\big]$.
\end{corollary}

\begin{proof}
By \cite{ChinRam2000,Hla2025a}, the best case optimal value of \eqref{minCorOnumIlpBest} can be expressed as the optimal value of the problem
\begin{align*}
 \max\ \Mid{c}^Tx+\Rad{c}^T|x| \st \Mid{A}^*x-\Rad{A}^*|x| \leq \ovr{b}.
\end{align*}
Since $\Mid{A}^*=\Mid{A}$ and $\Rad{A}^*=\Rad{A}+\omace{D}$, the problem takes the form of~\eqref{ofAvlp}.
\end{proof}

%%%%%
\paragraph{Worst case optimal value.} 
Obviously, the worst case realization is attained for $b=\uvr{b}$ and $D=\umace{D}$. Thus, can fix these values. 
A closed-form characterization of $\unum{f}$ is unknown in general; a special case of basis stability is discussed in the next section. That is why we focus on lower and upper bounds. 

%%%%%
\paragraph{Worst case optimal value: lower bound.} 

\begin{proposition}
We have
\begin{align*}
\unum{f} \geq \unum{f}^L,
\end{align*}
where
\begin{align}\label{dfUfL}
\unum{f}^L = \max\ \Mid{c}^Tx-\Rad{c}^T|x| \st \Mid{A}x+(\Rad{A}-\umace{D})|x| \leq \uvr{b}.
\end{align}
\end{proposition}

\begin{proof}
For every $x\in\R^n$, $A\in\imace{A}$ and $c\in\ivr{c}$ we have
$$c^Tx \geq \Mid{c}^Tx-\Rad{c}^T|x|$$
and also
\begin{align*}
 \Mid{A}x+(\Rad{A}-\umace{D})|x| \geq  Ax-\umace{D}|x|.
\end{align*}
Therefore
$$
f(A,\uvr{b},c,\umace{D}) \geq \unum{f}^L
$$
for every $A\in\imace{A}$ and $c\in\ivr{c}$, from which $\unum{f} \geq \unum{f}^L$ follows.
\end{proof}

As Proposition~\ref{propFLtight} below shows, under certain assumptions the lower bound $\unum{f}^L$ is tight and we have $\unum{f}=\unum{f}^L$. However, Example~\ref{exUfUflDiff} later on shows that $\unum{f} > \unum{f}^L$ may also happen.

\begin{proposition}\label{propFLtight}
Let $s^*$ be the sign of the optimal solution to \eqref{dfUfL}. If the optimal solution to
\begin{align}\label{maxPropFuLuEq}
f(A_{e,-s^*},c_{-s^*}) \coloneqq \max\ c^T_{-s^*}x \st A_{e,-s^*}x-\umace{D}|x| \leq \uvr{b}
\end{align}
has the same sign $s^*$, then  $\unum{f}=\unum{f}^L$.
\end{proposition}

\begin{proof}
We have
$$
\unum{f}^L \leq \unum{f} \leq f(A_{e,-s^*},c_{-s^*}) 
$$
since \eqref{maxPropFuLuEq} is a particular realization of the interval AVLP problem \eqref{intAvlp}. Thus, it is sufficient to show that 
$
\unum{f}^L = f(A_{e,-s^*},c_{-s^*}).
$ 
Indeed, this is true since
\begin{align*}
f(A_{e,-s^*},c_{-s^*})
 &= \max\ c^T_{-s^*}x \st A_{e,-s^*}x-\umace{D}|x| \leq \uvr{b},\ \diag(s^*)x\geq0 \\
% &= \max\ \Mid{c}^Tx-\Rad{c}^T|x| \st  
%  \begin{aligned}[t]
%   (\Mid{A}+\Rad{A}\diag(s^*))x-\umace{D})|x| &\leq \uvr{b},\\ \diag(s^*)x&\geq0 
%  \end{aligned} \\
 &= \max\ \Mid{c}^Tx-\Rad{c}^T|x| \\
 &\ \ \st  
   (\Mid{A}+\Rad{A}\diag(s^*))x-\umace{D})|x| \leq \uvr{b},\ \diag(s^*)x\geq0 \\
 &= \unum{f}^L.
\qedhere
\end{align*}
\end{proof}

%%%%%
\paragraph{Worst case optimal value: upper bound.} 

For the upper bound approach, we present an iterative improving method. For any $A\in\imace{A}$ and $c\in\ivr{c}$ the corresponding real-valued AVLP problem provides an upper bound, i.e., $f(A,\uvr{b},c,\umace{D})\geq\unum{f}$. So the idea of the iterative method is to seek for promising values from $\imace{A}$ and~$\ivr{c}$. 

We begin with the midpoint values $A\coloneqq\Mid{A}$ and $c\coloneqq\Mid{c}$. The corresponding optimal value $f(A,\uvr{b},c,\umace{D})$ provides an initial upper bound on~$\unum{f}$:
$$
\unum{f}^U \coloneqq f(A,\uvr{b},c,\umace{D}).
$$
Let $s$ be the sign of the computed optimal solution. In the orthant determined by $s$, the most promising choice is $A\coloneqq A_{e,-s}$ and $c\coloneqq c_{-s}$. For these values, we re-calculate the AVLP problem and improve the upper bound
$$
\unum{f}^U \coloneqq \min\braces*{\unum{f}^U,\,f(A,\uvr{b},c,\umace{D})}.
$$ 
We iterate this process until the upper bound $\unum{f}^U$ is not improved.

%%%%%
\paragraph{Worst case optimal value: examples.} 

Now, we present a few of examples pointing out specific properties of the best or worst case optimal values. The first example illustrates, among others, that there can be strict inequality $\unum{f}<\unum{f}^L$. 

\begin{example}\label{exUfUflDiff}
Consider the interval AVLP problem
\begin{align*}
\max\ x_2 \st -3\leq x_1\leq 3,\ x_2-|x_1|\leq 0,\ \inum{a} x_1+x_2\leq 3,
\end{align*}
where $\inum{a}=[-1,1]$. Figure~\ref{figExUfUflDiff} depicts the feasible set; the dashed lines correspond to particular instances of the inequality $\inum{a} x_1+x_2\le 3$ for $a\in\{-1,0,0.5,1\}$.
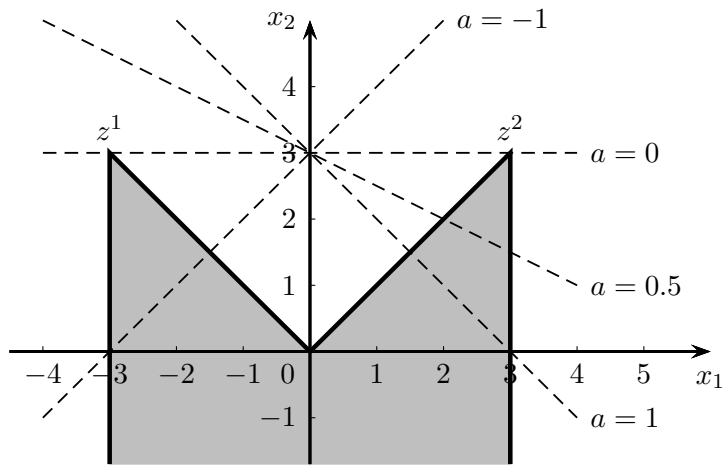
\begin{figure}[t]
\centering%\small\footnotesize
\psset{unit=5.3ex,arrowscale=1.5}
%\psset{unit=0.033\linewidth,arrowscale=1.5}
\begin{pspicture}(-5.2,-2.2)(6.65,5.5)
\newgray{mygray}{0.9}
\pspolygon[fillstyle=solid,fillcolor=lightgray,linecolor=lightgray,linewidth=0pt](-3,-1.7)(-3,3)(0,0)(3,3)(3,-1.7)(-3,-1.7)
\psaxes[ticksize=2pt,labels=all,ticks=all,showorigin=false,linewidth=1pt]{->}(0,0)(-4.5,-1.7)(6,5)
\psline[linewidth=1.7pt](-3,-1.7)(-3,3)(0,0)(3,3)(3,-1.7)
\uput[-135](-0.08,-0.06){$0$}
\uput[-90](6,-0.1){$x_1$}
\uput[-180](0,5.){$x_2$}
\uput[90](3,3){$z^2$}
\uput[90](-3,3){$z^1$}
\psline[linestyle=dashed, linewidth=0.7pt](-4,3)(4,3)
\uput[0](4,3){$a=0$}
\psline[linestyle=dashed, linewidth=0.7pt](-4,5)(4,1)
\uput[0](4,1){$a=0.5$}
\psline[linestyle=dashed, linewidth=0.7pt](-2,5)(4,-1)
\uput[0](4,-1){$a=1$}
\psline[linestyle=dashed, linewidth=0.7pt](-4,-1)(2,5)
\uput[0](2,5){$a=-1$}
\end{pspicture}
\caption{(Example~\ref{exUfUflDiff}) Illustration of the feasible set with different choices of $a\in\inum{a}$\label{figExUfUflDiff}}
\end{figure}

For $a\in[0,1]$, the optimal solution is $z^1=(-3,3)^T$ and the optimal value is~$3$. Similarly, for $a\in[-1,0]$, the optimal solution is $z^2=(3,3)^T$ and the optimal value is~$3$. Therefore, the optimal value $f(a)=3$ is constant on $a\in\inum{a}$ and we simply have $\unum{f}=3$.

However, the problem \eqref{dfUfL} reads as
\begin{align*}
\unum{f}^L = 
\max\ x_2 \st -3\leq x_1\leq 3,\ x_2-|x_1|\leq 0,\ x_2+|x_1|\le 3.
\end{align*}
It has two optimal solutions $(1.5,1.5)^T$ and $(-1.5,1.5)^T$, and the optimal value is $\unum{f}^L = 1.5$.

The difference between $\unum{f}$ and $\unum{f}^L$ can be even more enormous. Further, consider a variant of the example by including the inequality $x_2\geq 3$ into the constraints. Now, the optimal value $f(a)=3$ is still constant on $a\in\inum{a}$ and $\unum{f}=3$, while the problem \eqref{dfUfL} is infeasible, yielding $\unum{f}^L = -\infty$.

The upper bound iterative process works as follows. We choose $a\coloneqq\Mid{a}=0$ and compute the optimal value $f(a)=3$; the corresponding optimal solutions are both $z^1$ and $z^2$. We already achieved the worst case optimal value, so no further iterations improve it and we have tight bound $\unum{f}^U=\unum{f}$.
\end{example}

In contrast to the best case optimal value, the worst case optimal value $\unum{f}$ need not be achieved for some particular realization of the interval AVLP problem even when $\unum{f}$ is a finite value. We demonstrate it by the following example. Notice that in interval linear programming this is not the case; the  worst case optimal value is always attained for a particular realization, and under certain conditions the realization uses the endpoints of the interval coefficients \cite{Mra1998,Roh2006b}.

\begin{example}\label{exUfFinAttain}
Consider the interval AVLP problem
\begin{align*}
\max\ x_2 \st 
 &-1\leq x_1\leq 1,\\ 
 &-|x_1|\leq -1,\\ 
 &0\leq x_1+x_2\leq 1,\\ 
 &x_2\leq 1,\\ 
 &\inum{a} x_1+x_2\leq 0,
\end{align*}
where $\inum{a}=[0,1]$. The feasible set without the last constraint consists of a point $z^1=(-1,1)^T$ and the line segment joining points $(1,0)^T$ and $(1,-1)^T$. Figure~\ref{figExUfFinAttain} depicts the feasible set, where the dashed lines correspond to particular instances of the inequality $\inum{a} x_1+x_2\le 0$ for $a\in\{0.2,0.5,1\}$.
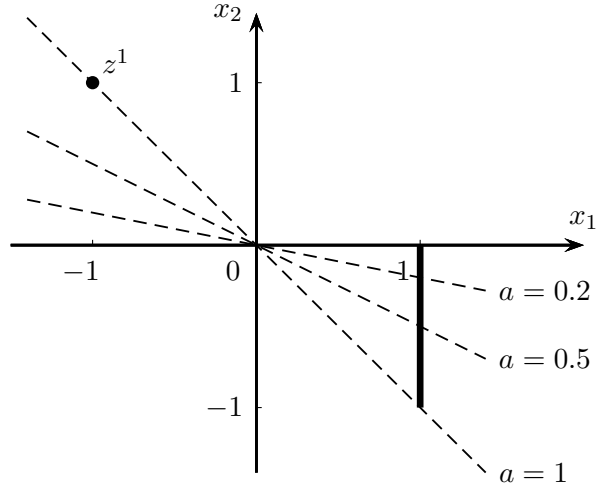
\begin{figure}[t]
\centering%\small\footnotesize
\psset{unit=13ex,arrowscale=1.5}
%\psset{unit=0.033\linewidth,arrowscale=1.5}
\begin{pspicture}(-1.8,-1.6)(2.3,1.6)
\newgray{mygray}{0.9}
\psaxes[ticksize=2pt,labels=y,ticks=all,showorigin=false,linewidth=1pt]{->}(0,0)(-1.5,-1.4)(2.0,1.43)
\qdisk(-1,1){2.5pt}
\psline[linewidth=2.5pt](1,0)(1,-1)
\uput[-135](-0.04,-0.04){$0$}
\uput[-135](-0.9,-0.04){$-1$}
\uput[-135](1.0,-0.04){$1$}
\uput[90](2.0,0.02){$x_1$}
\uput[-180](0,1.43){$x_2$}
\uput[45](-1,1){$z^1$}
\psline[linestyle=dashed, linewidth=0.7pt](-1.4,1.4)(1.4,-1.4)
\uput[0](1.4,-1.4){$a=1$}
\psline[linestyle=dashed, linewidth=0.7pt](-1.4,0.7)(1.4,-0.7)
\uput[0](1.4,-0.7){$a=0.5$}
\psline[linestyle=dashed, linewidth=0.7pt](-1.4,0.28)(1.4,-0.28)
\uput[0](1.4,-0.28){$a=0.2$}
\end{pspicture}
\caption{(Example~\ref{exUfFinAttain}) Illustration of the feasible set with different choices of $a\in\inum{a}$\label{figExUfFinAttain}}
\end{figure}

For $a\in[0,1)$, the optimal solution is $(1,-a)^T$ and the optimal value is~$-a$. For $a=1$, the optimal solution is $z^1$ and the optimal value is~$1$. This means that the worst case optimal value is $\unum{f}=-1$, but it is not attained for a particular realization of $a\in\inum{a}$, as it is the limit value of $f(a)$ when $a\to 1$.

To compute the lower bound $\unum{f}^L$, we basically replace the last inequality $\inum{a} x_1+x_2\leq 0$ by $0.5x_1+0.5|x_1|+x_2\leq 0$. The optimal value  of the problem is $\unum{f}^L=-1$, so the lower bound is tight.

How works the upper bound iterative process for this example? We select $a\coloneqq\Mid{a}=0.5$ and compute the optimal value $f(a)=-0.5$. The optimal solution is $(1,-0.5)^T$, so its sign vector $s=(1,-1)^T$ causes the update $a\coloneqq1$. For this selection, the optimal value increases. Thus, the iterations are terminated and we obtain the upper bound $\unum{f}^U=-0.5>-1=\unum{f}$.
\end{example}

Further, we illustrate by an example that $\unum{f}=-\infty$ does no necessarily imply that there is an infeasible realization (the converse implication is always true). This is again in contrast to interval linear programming, where $\unum{f}=-\infty$ if and only if there exists an infeasible realization \cite{Hla2013b,Hla2025a}. %vyplyva z Hla2013b

\begin{example}\label{exUfInfInfeas}
Consider the interval AVLP problem
\begin{align*}
\max\ x_2 \st 
 &-1\leq x_1\leq 1,\\ 
 &-|x_1|\leq -1,\\ 
% &x_1+x_2\leq 1,\\ 
 &3x_1+2x_2\leq -1,\\ 
 &-|x_2|\leq -1,\\ 
 &x_1+\inum{a} x_2\leq -1,
\end{align*}
where $\inum{a}=[0,1]$. Wihout the last (interval-valued) constraint, the feasible set consists of a point $z^1=(-1,1)^T$ and two rays emerging from points and $(-1,-1)^T$ and $(1,-2)^T$ in the direction of $(0,-1)^T$. Figure~\ref{figExUfInfInfeas} depicts the feasible set, where the dashed lines correspond to particular instances of the inequality $x_1+\inum{a} x_2\le -1$ for $a\in\{0,0.5,1\}$.
\begin{figure}[t]
\centering%\small\footnotesize
\psset{xunit=9.5ex,yunit=7ex,arrowscale=1.5}
%\psset{unit=0.033\linewidth,arrowscale=1.5}
\begin{pspicture}(-2.0,-5.0)(2.9,2.2)
\newgray{mygray}{0.9}
\psaxes[ticksize=2pt,labels=all,ticks=all,showorigin=false,linewidth=1pt]{->}(0,0)(-1.9,-4.8)(2.5,1.6)
\qdisk(-1,1){2.5pt}
\psline[linewidth=2.5pt](1,-2)(1,-4.8)
\psline[linewidth=2.5pt](-1,-1)(-1,-4.8)
\uput[-135](-0.04,-0.04){$0$}
%\uput[-135](-0.9,-0.04){$-1$}
%\uput[-135](1.0,-0.04){$1$}
\uput[90](2.5,0.02){$x_1$}
\uput[-180](0,1.6){$x_2$}
\uput[45](-1,1){$z^1$}
\psline[linestyle=dashed, linewidth=0.7pt](-1.8,0.8)(1.6,-2.6)
%\uput[0](1.6,-2.6){$a=1$}
\uput[180](-1.8,0.8){$a=1$}
\psline[linestyle=dashed, linewidth=0.7pt](-1,1.6)(-1,-4.8)
%\uput[180](-1,-4.8){$a=0$}
\uput[110](-1,1.6){$a=0$}
\psline[linestyle=dashed, linewidth=0.7pt](-1.7,1.4)(1.4,-4.8)
%\uput[0](1.4,-4.8){$a=0.5$}
\uput[145](-1.7,1.4){$a=0.5$}
\end{pspicture}
\caption{(Example~\ref{exUfInfInfeas}) Illustration of the feasible set with different choices of $a\in\inum{a}$\label{figExUfInfInfeas}}
\end{figure}
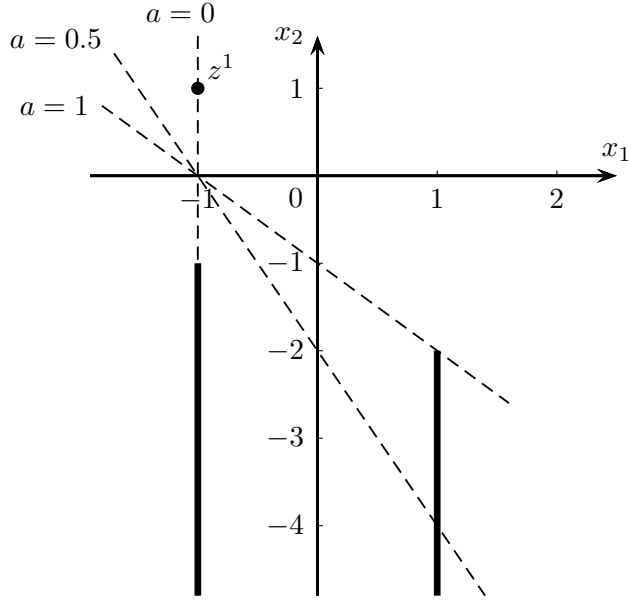

For $a=0$, the optimal solution is $z^1$ and the optimal value is~$1$, while for $a\in(0,1]$, the optimal solution is $(1,-2/a)^T$ and the optimal value is $-2/a$. Thus, $\unum{f}=-\infty$, but this value is achieved as a limit and not attained as an optimal value of some realization.

In this example, the lower bound is again tight, $\unum{f}^L=-\infty$. In contrast, the upper bound is poor now. The initial value $a\coloneqq\Mid{a}=0.5$ yields $f(a)=-4$. The next iteration selects $a\coloneqq 0$, producing $f(a)=1$. Thus, the upper bounds is $\unum{f}^U=-4$.
\end{example}

%%%%%%%%%%%%%%%%%%%%%%%%%%%%%%%%%%%%%%%%%%%%%%%%%%%%%%%%%%%%%%%
\section{Basis stability}

%%%%%
\paragraph{Basis in LP.} 
For an LP problem \eqref{lp}, a basis is an index set $B$ such that $A_B$ is nonsingular. Basis $B$ is optimal if the following two conditions hold
\begin{align}
\label{lpBasis1}
A_NA_B^{-1}b_B&\leq b_N,\\
\label{lpBasis2}
A_B^{-T}c&\geq0,
\end{align}
where $N\coloneqq \seznam{n}\setminus B$ is the set of nonbasic indices. The optimal solution corresponding to basis $B$ is $x=A_B^{-1}b_B$. 
Further, basis $B$ is nondegenerate if inequality \eqref{lpBasis2} holds strictly.

%%%%%
\paragraph{Basis stability in interval LP.} 
In interval linear programming, many problems becomes tractable if there is a basis that is optimal for every realization of interval coefficients. This is also the case here to some extent. 
Let $B$ be a basis of the LP problem \eqref{lp}. Define $\imace{A}^*$ as in Corollary~\ref{corOnumIlpBest},
$$
\imace{A}^* \coloneqq \big[\Mid{A}-\Rad{A}-\omace{D},\,\Mid{A}+\Rad{A}+\omace{D}\big],
$$
and, consider the corresponding interval LP problem
\begin{align}\label{ilp}
\max\ \ivr{c}^Tx \st  \imace{A}^* x \leq \ivr{b}.
\end{align}
This interval LP problem can be considered as a relaxation of our interval AVLP problem~\eqref{intAvlp}. Indeed, in Corollary~\ref{corOnumIlpBest}, we showed that both share the same best case optimal value. Moreover, they have also the same united solution set. It is the set of solutions that are feasible for at least one realization of interval values, and for both interval systems $\imace{A}^* x \leq \ivr{b}$ and $\imace{A}x-\imace{D}|x|\leq \ivr{b}$ they are characterized by the real-valued system (see the proof of Proposition~\ref{propOfAvlp})
$$
\Mid{A}x-(\Rad{A}+\omace{D})|x| \leq \ovr{b}.
$$

The interval LP problem \eqref{ilp} is called (nondegenerate) $B$-stable \cite{Bee1978,Hla2014a,Jan1988,Kon2001,Kra1975} if $B$ is a (nondegenerate) optimal basis to every realization, that is, to \eqref{lp} for each $A^*\in\imace{A}^*$.

%%%%%
\paragraph{Basis stability in interval AVLP.} 

We say that the interval AVLP problem \eqref{intAvlp} is \emph{$B$-stable} if the corresponding interval LP problem \eqref{ilp} is $B$-stable.

The best case optimal value then can be computed by solving only one LP problem~\eqref{maxPropBstabOf}, whereas the worst case optimal value analysis is rather more complicated.

\begin{proposition}
Under $B$-stability, we have
\begin{align}\label{maxPropBstabOf}
\onum{f} = \max\ \ovr{b}_B^Ty \st  
 (\Mid{A}-\Rad{A}-\omace{D})_B^T y \leq \ovr{c},\ 
 (\Mid{A}+\Rad{A}+\omace{D})_B^T y \geq \uvr{c},\ y\geq0.
\end{align}
\end{proposition}

\begin{proof}
Recall from Corollary~\ref{corOnumIlpBest} that $\onum{f}$ is the same as the best case optimal value of the interval LP problem~\eqref{ilp}. The dual problem to \eqref{ilp} reads as
\begin{align*}
\min\ \ivr{b}^Ty \st  (\imace{A}^*)^T y =\ivr{c},\ y\geq0.
\end{align*}
Since it is also basis stable, $\onum{f}$ is its worst case optimal value. By Beeck~\cite{Bee1978}, we obtain
\begin{align*}
\onum{f} = \max\ \ovr{b}_B^Ty \st  
 (\umace{A}^*)_B^T y \leq \ovr{c},\ 
 (\omace{A}^*)_B^T y \geq \uvr{c},\ y\geq0,
\end{align*}
which is equivalent to~\eqref{maxPropBstabOf}.
\end{proof}

Regarding the worst case optimal value~$\unum{f}$, we first derive a characterization of it and its lower bound $\unum{f}^L$. Then, we show that both values are equal and can be expressed by means of solving one absolute value equation system.

\begin{proposition}\label{propBstabUf}
Under $B$-stability, we have
\begin{align}\label{minPfPropBstabUf}
\unum{f} =  
 \min\ \Mid{c}^Tx-\Rad{c}^T|x| \st (\Mid{A})_Bx+(\Rad{A}-\umace{D})_B|x| \geq \uvr{b}_B.
\end{align}
\end{proposition}

\begin{proof}
Let $A\in\imace{A}$ and $c\in\ivr{c}$. Then the optimal value of the AVLP problem
\begin{align*}
f(A,c)=\max\ c^Tx \st Ax-\umace{D}|x|\leq \uvr{b}
\end{align*}
is the best case optimal value of the interval LP problem
\begin{align*}
\max\ c^Tx \st [A-\umace{D},A+\umace{D}]x\leq \uvr{b},
\end{align*}
or, equivalently, the worst case optimal value of its dual
\begin{align*}
\min\ \uvr{b}^Ty \st [A-\umace{D},A+\umace{D}]^Ty=c,\ y\geq0.
\end{align*}
Due to $B$-stability, the optimal value can be expressed as
\begin{align*}
f(A,c)
 &= \max\ \uvr{b}_B^T\tilde{A}_B^{-T}c \st  \tilde{A}\in[A-\umace{D},A+\umace{D}] \\
 &= \max\ \uvr{b}_B^Ty \st \tilde{A}_B^Ty=c,\ y\geq0,\ \tilde{A}\in[A-\umace{D},A+\umace{D}].
\end{align*}
Using the Oettli-Prager theorem \cite{OetPra1964} characterizing the solution set of an interval system of linear equations (see also \cite{Hla2013b,Roh2006a}), the problem takes the form of an LP problem
\begin{align*}
f(A,c)
 = \max\ \uvr{b}_B^Ty \st (A-\umace{D})_B^Ty\leq c\leq (A+\umace{D})_B^Ty,\ y\geq0.
\end{align*}
Its dual problem reads as
\begin{align}\nonumber
f(A,c)
 &=\min\ c^Tx_1-c^Tx_2 \st (A-\umace{D})_Bx_1-(A+\umace{D})_Bx_2\geq \uvr{b}_B,\ x_1,x_2\geq0\\
 &=\min\ c^T(x_1-x_2) \st A_B(x_1-x_2)-\umace{D}_B(x_1+x_2)\geq \uvr{b}_B,\ x_1,x_2\geq0.
 \label{minPfPropBstabUfNonneg}
\end{align}
We claim that this dual problem has an equivalent form
\begin{align}\label{minPfPropBstabUfAbs}
f(A,c)
 =\min\ c^Tx \st A_Bx-\umace{D}_B|x|\geq \uvr{b}_B.
\end{align}
Indeed, if $x$ is a feasible solution to \eqref{minPfPropBstabUfAbs}, then $x_1\coloneqq\max(x,0)$ and $x_2\coloneqq\max(-x,0)$ form a feasible solution to \eqref{minPfPropBstabUfNonneg} with the same objective value. Conversely, let $x_1,x_2$ be feasible to \eqref{minPfPropBstabUfNonneg} and define $z\coloneqq\min(x_1,x_2)\geq0$. Then $\tilde{x}_1\coloneqq x_1-z$ and $\tilde{x}_2\coloneqq x_2-z$ form also a feasible solution with the same objective value and satisfying the complementarity $(\tilde{x}_1)^T\tilde{x}_2=0$. Then $x\coloneqq \tilde{x}_1-\tilde{x}_2$ is feasible to \eqref{minPfPropBstabUfAbs} with the same objective value. 

Now, 
\begin{align*}
\unum{f}
 &= \min\ f(A,c) \st A\in\imace{A},\,c\in\ivr{c} \\
 &= \min\ \parentheses*{\min_{c\in\ivr{c}}\, c^Tx} 
  \st A_Bx-\umace{D}_B|x|\geq \uvr{b}_B,\ A\in\imace{A}.
\end{align*}
Using the Gerlach theorem \cite{Ger1981} describing the solution set of interval system of linear inequalities (cf.\ \cite{Hla2013b,Roh2006a}), we get~\eqref{minPfPropBstabUf}.
\end{proof}

Obviously, if $\Rad{c}=0$ and $(\Rad{A})_B\leq\umace{D}_B$, then \eqref{minPfPropBstabUf} can be expressed as an LP problem and the computation of $\unum{f}$ becomes a polynomial time problem.

It turns out important also to derive a specific formula for the lower bound $\unum{f}^L$ on~$\unum{f}$. We present it in the following proposition. We can observe that, in the case of basis stability, $\unum{f}^L$ is efficiently solvable by means of linear programming  provided $\Rad{c}=0$ and $(\Rad{A})_B\geq\umace{D}_B$.

\begin{proposition}
Under $B$-stability, we have
\begin{align}\label{minPfPropBstabUfL}
\unum{f}^L
 =\max\ \Mid{c}^Tx-\Rad{c}^T|x| \st (\Mid{A})_Bx+(\Rad{A}-\umace{D})_B|x| \leq \uvr{b}_B.
\end{align}
\end{proposition}

\begin{proof}
Reformulate the problem \eqref{dfUfL} as
\begin{align*}
\unum{f}^L = \max\ \Mid{c}^Tx-\Rad{c}^Tz 
 \st \Mid{A}x+\Rad{A}z-\umace{D}|x| \leq \uvr{b},\ -z\leq x\leq z.
\end{align*}
Then $\unum{f}^L$ represents the best case optimal value of the interval LP problem
\begin{align*}
\max\ \Mid{c}^Tx-\Rad{c}^Tz
 \st [\Mid{A}-\umace{D},\Mid{A}+\umace{D}]x+\Rad{A}z \leq \uvr{b},\ -z\leq x\leq z,
\end{align*}
which we can expressed as 
\begin{align*}
\max\ \Mid{c}^Tx-\Rad{c}^T|x|
 \st [\Mid{A}-\umace{D},\Mid{A}+\umace{D}]x+\Rad{A}|x| \leq \uvr{b}.
\end{align*}
Let $\tilde{A}\in[\Mid{A}-\umace{D},\Mid{A}+\umace{D}]$ and consider the particular realization
\begin{align*}
\unum{f}^L(\tilde{A}) = \max\ \Mid{c}^Tx-\Rad{c}^T|x|
 \st \tilde{A}x+\Rad{A}|x| \leq \uvr{b}.
\end{align*}
This value represents the worst case optimal value of the interval LP problem
\begin{align*}
\max\ \ivr{c}^Tx  \st [\tilde{A}-\Rad{A},\tilde{A}+\Rad{A}]x \leq \uvr{b},
\end{align*}
or, equivalently, the best case optimal value of the dual
\begin{align*}
\min\ \uvr{b}^Ty  \st [\tilde{A}-\Rad{A},\tilde{A}+\Rad{A}]^Ty=\ivr{c},\ y\geq0.
\end{align*}
In view of $B$-stability and using again the Oettli-Prager theorem, we can write 
\begin{align*}
\unum{f}^L(\tilde{A}) 
 &=\min\ \uvr{b}_B^TA'{}_B^{-T}c 
    \st A'\in[\tilde{A}-\Rad{A},\tilde{A}+\Rad{A}],\,c\in\ivr{c} \\
 &=\min\ \uvr{b}_B^Ty 
    \st A'{}_B^Ty=c,\ y\geq0,\ A'\in[\tilde{A}-\Rad{A},\tilde{A}+\Rad{A}],\,c\in\ivr{c} \\
 &=\min\ \uvr{b}_B^Ty 
    \st (\tilde{A}+\Rad{A})^T_By\geq\uvr{c},\ (\tilde{A}-\Rad{A})^T_By\leq\ovr{c},\ y\geq0.
\end{align*}
By duality in linear programming,
\begin{align*}
\unum{f}^L(\tilde{A}) 
 ={}&\max\ \uvr{c}^Tx_1-\ovr{c}^Tx_2 \\
    &\stl (\tilde{A}+\Rad{A})_Bx_1-(\tilde{A}-\Rad{A})_Bx_2\leq\uvr{b}_B,\ x_1,x_2\geq0 \\
 ={}&\max\ \Mid{c}^T(x_1-x_2)-\Rad{c}^T(x_1+x_2) \\
    &\stl \tilde{A}_B(x_1-x_2)+(\Rad{A})_B(x_1+x_2)\leq\uvr{b}_B,\ x_1,x_2\geq0.
\end{align*}
Using the same reasoning that showed equivalence between \eqref{minPfPropBstabUfNonneg} and \eqref{minPfPropBstabUfAbs}, we obtain that
\begin{align*}
\unum{f}^L(\tilde{A}) 
 = \max\ \Mid{c}^Tx-\Rad{c}^T|x| \st \tilde{A}_Bx+(\Rad{A})_B|x| \leq \uvr{b}_B.
\end{align*}
Now,
\begin{align*}
\unum{f}^L
 &=\max\ \unum{f}^L(\tilde{A}) \st \tilde{A}\in[\Mid{A}-\umace{D},\Mid{A}+\umace{D}] \\
 &=\max\ \Mid{c}^Tx-\Rad{c}^T|x| \st \tilde{A}_Bx+(\Rad{A})_B|x| \leq \uvr{b}_B,\ 
    \tilde{A}\in[\Mid{A}-\umace{D},\Mid{A}+\umace{D}] \\
 &=\max\ \Mid{c}^Tx-\Rad{c}^T|x| \st (\Mid{A})_Bx+(\Rad{A}-\umace{D})_B|x| \leq \uvr{b}_B.
\qedhere
\end{align*}
\end{proof}

It turns out that the inequalities in \eqref{minPfPropBstabUf} and \eqref{minPfPropBstabUfL} are attained as equations, which results in Theorem~\ref{thmBstabUfLEq}, and then a more specific result in Theorem~\ref{thmBstabUf}.

\begin{theorem}\label{thmBstabUfLEq}
Under nondegenerate $B$-stability, we have
\begin{align}\label{minPfPropBstabUfLEq}
\unum{f}^L
 =\max\ \Mid{c}^Tx-\Rad{c}^T|x| \st (\Mid{A})_Bx+(\Rad{A}-\umace{D})_B|x| = \uvr{b}_B.
\end{align}
\end{theorem}

\begin{proof}
The value of $\unum{f}^L$ is attained at some optimal solution $x^*$ of problem~\eqref{minPfPropBstabUfL}. Indeed, basis stability ensures that the problem is not unbounded and also that it is feasible (even more, we will see later in \eqref{gaveThmUf} that a solution exists such that all inequalities are satisfied as equations).

In case the optimal solution is not unique, take that one which maximizes the number of inequalities satisfied as equations. Denote 
%$s^*\coloneqq\sgn(x^*)$, 
$K\coloneqq\{k\mmid x_k^*=0\}$, and
\begin{align*}
\mna{S} &\coloneqq \{s\in\{\pm1\}^n\mmid \diag(s)x^*\geq0\},\\
I &\coloneqq 
 \braces*{i\mmid \parentheses*{(\Mid{A})_Bx^*+(\Rad{A}-\umace{D})_B|x^*|}_i = (\uvr{b}_B)_i}.
\end{align*}
The set $\mna{S}$ corresponds to all orthants containing the point~$x^*$, and $I$ is the set of active constraints for~$x^*$. 
Let $s\in\mna{S}$, and consider the problem \eqref{minPfPropBstabUfL} restricted to orthant~$s$:
\begin{align}\label{pfPropBstabUfLEqMaxDs}
% \max\ c_{-s}^Tx \st (A_{e,-s}-\umace{D})_Bx \leq \uvr{b}_B,\ \diag(s)x\geq0.
 \max\ c_{-s}^Tx \st 
   \parentheses*{\Mid{A}+(\Rad{A}-\umace{D})\diag(s)}_Bx \leq \uvr{b}_B,\ \diag(s)x\geq0.
\end{align}
Since $x^*$ is its optimal solution, the objective vector is a nonnegative combination of the normals of active constraints, i.e.,
\begin{align}\label{pfPropBstabUfLEqS}
c_{-s} = \big(\Mid{A}+(\Rad{A}-\umace{D}\big)\diag(s))^T_I u^s - \sum_{k\in K}s_ke_k v^s_k
\end{align}
for some $u^s\geq0$ and $v^s_K\geq0$. 

Since the origin lies in the convex hull of the points $\sum_{k\in K}s_ke_i v^s_i$, $s\in\mna{S}$, there is their convex combination producing zero. That is, there are $\lambda_s\geq0$, $\sum_{s\in\mna{S}}\lambda_s=1$, such that
\begin{align}\label{pfPropBstabUfLEqConv}
\sum_{s\in\mna{S}}\lambda_s c_{-s}
 = \sum_{s\in\mna{S}}\lambda_s \big(\Mid{A}+(\Rad{A}-\umace{D})\diag(s)\big)^T_I u^s.
\end{align}

Denote
\begin{align*}
c^* &\coloneqq \sum_{s\in\mna{S}}\lambda_s c_{-s},\\
u^* &\coloneqq \sum_{s\in\mna{S}}\lambda_s u^s \geq 0.
\end{align*}
Then there is $z\in[-1,1]^n$ such that equation \eqref{pfPropBstabUfLEqConv} reads as
\begin{align}\label{pfPropBstabUfLEqCstar}
c^* = \big(\Mid{A}+(\Rad{A}+\umace{D})\diag(z)\big)^T_I u^*.
\end{align}
Due to $B$-stability, each realization of $\ivr{c}$ is a nonnegative combination of rows of any realization of $\imace{A}^*_B$. That is,
\begin{align}\label{pfPropBstabUfLEqBstab}
\forall c\in\ivr{c},\ \forall A^*_B\in\imace{A}^*_B,\ \exists u\geq0: 
 c=(A^*)_B^Tu.
\end{align}
In \eqref{pfPropBstabUfLEqCstar}, the combination uses only the rows indexed by~$I$, so that the other coefficients $u_i$ are zero. This is in contradiction with the nondegeneracy of basis~$B$. Therefore, $I=\seznam{n}$.
\end{proof}

In a similar way we tightened the inequalities to equations in the optimization problem for $\unum{f}^L$, we can obtain equality constraints for~$\unum{f}$.

\begin{corollary}\label{corBstabUf}
Under nondegenerate $B$-stability, we have
\begin{align}\label{minPfPropBstabUfEq}
\unum{f} =  
 \min\ \Mid{c}^Tx-\Rad{c}^T|x| \st (\Mid{A})_Bx+(\Rad{A}-\umace{D})_B|x| = \uvr{b}_B.
\end{align}
\end{corollary}

\begin{proof}
First, we transform the problem \eqref{minPfPropBstabUf} to the form of
\begin{align*}
\unum{f} =  
 -\max\ \Mid{c}^Tx+\Rad{c}^T|x| \st (\Mid{A})_Bx-(\Rad{A}-\umace{D})_B|x| \leq -\uvr{b}_B,
\end{align*}
and then we use the same reasoning as in the proof of Theorem~\ref{thmBstabUfLEq}.
\end{proof}

The feasible sets in \eqref{minPfPropBstabUfLEq} and \eqref{minPfPropBstabUfEq} are described by the GAVE system
\begin{align}\label{gaveThmUf}
(\Mid{A})_Bx+(\Rad{A}-\umace{D})_B|x| = \uvr{b}_B.
\end{align}
$B$-stability implies regularity of the interval matrix $\imace{A}^*_B$, which in turn causes that \eqref{gaveThmUf} has a unique solution; see \cite{KumDee2024,Roh2009,WuShe2021}. We denote this solution as $x^*$ and show that it is the minimizer for which $\unum{f}$ is achieved.

\begin{theorem}\label{thmBstabUf}
Under nondegenerate $B$-stability, the unique solution $x^*$ of \eqref{gaveThmUf} is the optimal solution of problems \eqref{dfUfL}, \eqref{minPfPropBstabUf}, \eqref{minPfPropBstabUfL}, \eqref{minPfPropBstabUfLEq} and \eqref{minPfPropBstabUfEq}. Further,
\begin{align*}
\unum{f} = \unum{f}^L =  \Mid{c}^Tx^*-\Rad{c}^T|x^*|.
\end{align*}
\end{theorem}

\begin{proof}
%Proposition~\ref{propBstabUfLEq}
%From \eqref{minPfPropBstabUfLEq} and \eqref{minPfPropBstabUfEq}, we get
Since the equation \eqref{gaveThmUf} has the unique solution~$x^*$, we simply have from Theorem~\ref{thmBstabUfLEq} and Corollary~\ref{corBstabUf} that $\unum{f} = \unum{f}^L$, and both values are the objectives at~$x^*$.
\end{proof}

Once we have the solution $x^*$, we also have the worst case optimal value $\unum{f}$. However, how difficult is to compute~$x^*$? Solving GAVE systems is an NP-hard problem in general, but the complexity is unknown when the interval matrix $\imace{A}^*_B$ is regular. In contrast, if one of the sufficient conditions \eqref{sufCondRegBeeck} or \eqref{sufCondRegRexRohn} is satisfied, then the GAVE is solvable in polynomial time~\cite{ManMey2006,ZamHla2021a}.

As we mentioned earlier, $\unum{f}$ is achieved for $b=\uvr{b}$ and $D=\umace{D}$. However, it is hard to specify $A$ and $c$ in general. However, in the case of basis stability, we can easily deduce that $\unum{f}$ is achieved for
\begin{align*}
c\coloneqq c_{-s},\ \ 
A_B\coloneqq (A_{e,-s})_B,
\end{align*}
where $s\coloneqq\sgn(x^*)$ and $x^*$ is the solution of~\eqref{gaveThmUf}. The nonbasic submatrix $A_N$ can take any value from $\imace{A}_N$.

%%%%%%%
%\paragraph{Alternative approach.} 
%By \cite{HlaHar2023au}, 

\begin{example}\label{exBstabUf}
Let
\begin{align*}
A=\begin{pmatrix} 1 & 1 \\ -2 & 4 \\ -6 & 2 \\ 4 & -7 \end{pmatrix},\ \ 
b=\begin{pmatrix} 12 \\ 18 \\ 36 \\ 26 \end{pmatrix},\ \ 
c=\begin{pmatrix} 1 \\ 2 \end{pmatrix},\ \ 
D=\begin{pmatrix} 0 & 0 \\ 1 & 1 \\ 1 & 1 \\ 1 & 1 \end{pmatrix}.
\end{align*}
The feasible region is depicted in Figure~\ref{figExBstabUf} in grey color and the optimal solution is the top most vertex $x^0=(3,9)^T$. The corresponding optimal value is $c^Tx^0=21$ and the optimal basis $B=(1,2)$.  
\begin{figure}[t]
\centering%\small\footnotesize
\psset{unit=2.5ex,arrowscale=1.5}
%\psset{unit=0.033\linewidth,arrowscale=1.5}
\begin{pspicture}(-10,-9)(12.5,12)
\newgray{mygray}{0.9}
\pspolygon[fillstyle=solid,fillcolor=mygray,linecolor=mygray,linewidth=0pt]
(-8,-1.333)(-7.2,0)(-6.4285,3.8572)(0,6)(3,9)(10.91,1.091)(8.667,0)(0,-4.333)(-4.4,-8)(-8,-8)
\psline[linewidth=0.75pt]
(-8,-1.333)(-7.2,0)(-6.4285,3.8572)(0,6)(3,9)(10.91,1.091)(8.667,0)(0,-4.333)(-4.4,-8)

\psaxes[ticksize=2pt,labels=all,ticks=all,Dx=2,Dy=2,showorigin=false,linewidth=1pt]{->}(0,0)(-9,-8.5)(11.5,11)
\uput[-135](-0.04,-0.04){$0$}
\uput[-90](11.5,0.02){$x_1$}
\uput[-180](0,11){$x_2$}
\psline[linewidth=1.25pt]
(2.9439,9.6878)(2.3729,9.0557)(3.0445, 8.3841)(3.6756,8.956)(2.9439,9.6878)
\psline[linewidth=1.25pt]
(2.2839,4.4845)(2.2839,11.4356)(9.7894,11.4356)(9.7894,4.4845)(2.2839,4.4845)
\uput[0](9.7894,7){$\ivr{x}$}
%\qdisk(3.0445,8.38408){2.5pt}
\end{pspicture}
\caption{(Example~\ref{exBstabUf}) The grey area is the feasible set of the nominal values; the rectangle is an enclosure $\ivr{x}$ of the optimal values; and the diamond-like shape is the real optimal solution set\label{figExBstabUf}}
\end{figure}
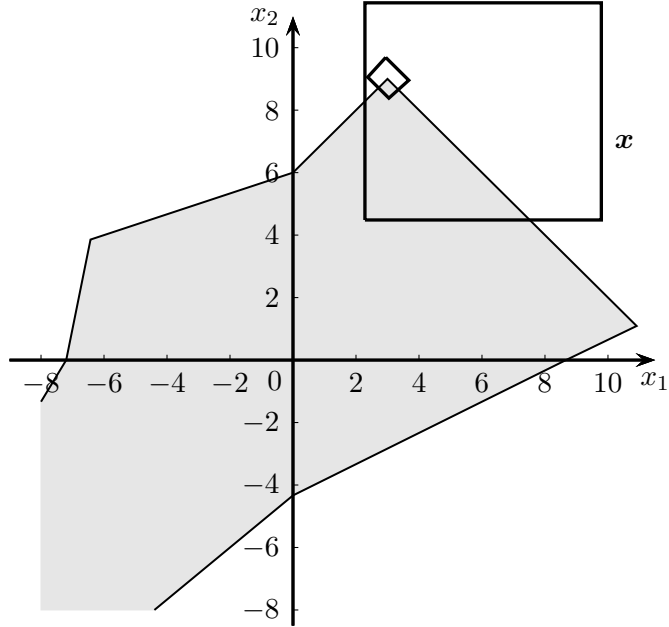

Consider now $5\%$ uncertainty in the constraint matrix~$A$. Thus, we replace the matrix by the interval matrix $\imace{A}=[A-0.05|A|,A+0.05|A|]$ and arrive at the interval AVLP problem~\eqref{intAvlp}. Let us check for basis stability with basis~$B$ by using the method from~\cite{Hla2014a}. The sufficient conditions succeed and verify basis stability. In addition, they provide an enclosure for the optimal solutions of all realizations of~\eqref{intAvlp} in the form of an interval enclosure
\begin{align*}
\ivr{x}=([ 2.2839,    9.7894],\, [ 4.4845,   11.4356] )^T.
\end{align*}
Since the intervals are nonnegative, we can also easily compute the true set of all optimal solutions, which is a convex polyhedron in this case. Concretely, it is the convex hull of the points
\begin{align*}
x^1 &= (2.9439,9.6878)^T,\ \ 
x^2 = (2.3729,9.0557)^T,\\ 
x^3 &= (3.0445, 8.3841)^T,\ \ 
x^4 = (3.6756,8.956),
\end{align*}
which has the form of a diamond-looking polygon in the figure. From the figure we see that the interval vector $\ivr{x}$ significantly overestimates the true set of optimal solutions. This is caused by two aspects: first, it is the interval arithmetic used to solve interval systems during the verification of basis stability, and, second, it is the nonconvexity of the feasible set caused by matrix $D$, which was relaxed as interval linear constraints. The good news is, however, that despite this overestimation we were able to determine the sign of the optimal solutions and consequently the whole optimal solution set.

Once we verified $B$-stability, we can compute the best and worst case optimal values $\onum{f}$ and $\unum{f}$, respectively. Solving the LP problem~\eqref{maxPropBstabOf}, we obtain $\onum{f}=22.319$. Notice that this is the objective value at point $x^1$, so this point is the optimum corresponding to the best case realization. 
Regarding the worst case optimal value, we solve the absolute value system \eqref{gaveThmUf}. It has the unique solution $x^3$, yielding the value of $\unum{f}=19.813$. 
\end{example}

%%%%%%%%%%%%%%%%%%%%%%%%%%%%%%%%%%%%%%%%%%%%%%%%%%%%%%%%%%%%%%%
\section{Conclusion}

This paper is the first contribution to the problem of absolute value liner programming with interval coefficients. We focused in particular on the problem of determining the best and worst optimal values, which is one of the basic problems in interval-valued optimization. We derived a formula for the best case, and for the worst case we proposed lower and upper bound methods. In case of basis stability, the problems become seemingly simpler, but even in this case there is some computational effort.

Since this paper initializes the research focus on interval-valued AVLP problems, several topics remain open for future investigation. For instance, the description and approximation of the set of optimal solutions. Also an adaptation of other interval-valued techniques, such as fixed ordering of intervals \cite{GhoDeb2020}, could be analysed in the context of our problem.

During the course of this research, we identified several open questions: First, no closed-form characterization of $\unum{f}$ is currently known.  Second, the computational complexity of computing $\unum{f}$ under (nondegenerate) $B$-stability is unknown, too. Third, is Theorem~\ref{thmBstabUfLEq} valid under possibly degenerate $B$-stability, or nondegeneracy is essential here?

%%%%%%%
\paragraph{Acknowledgments.} 
The author was supported by the Czech Science Foundation Grant 25-15714S.

%%%%%%%%%%%%%%%%%%%%%%%%%%%%%%%%%%%%%%%%%%%%%%%%%%%%%%%%%%%%%%%
% REFERENCES
%%%%%%%%%%%%%%%%%%%%%%%%%%%%%%%%%%%%%%%%%%%%%%%%%%%%%%%%%%%%%%%

\bibliographystyle{abbrv}
\bibliography{int_avlp}

\newcommand{\SortNoop}[1]{}
\begin{thebibliography}{10}

\bibitem{Beeck1975}
H.~Beeck.
\newblock {Zur Problematik der H\"{u}llenbestimmung von
  Intervallgleichungssystemen}.
\newblock In K.~Nickel, editor, {\em Interval Mathemantics: Proceedings of the
  International Symposium on Interval Mathemantics}, volume~29 of {\em LNCS},
  pages 150--159, Berlin, 1975. Springer.
\newblock in German.

\bibitem{Bee1978}
H.~Beeck.
\newblock Linear programming with inexact data.
\newblock technical report TUM-ISU-7830, Technical University of Munich,
  Munich, 1978.

\bibitem{ChenHan2025u}
Y.~Chen and D.~Han.
\newblock Absolute value inequalities.
\newblock {\em Optim. Lett.}, 2025.
\newblock in press.

\bibitem{ChinRam2000}
J.~W. Chinneck and K.~Ramadan.
\newblock Linear programming with interval coefficients.
\newblock {\em J. Oper. Res. Soc.}, 51(2):209--220, 2000.

\bibitem{GabMur2010}
V.~Gabrel, C.~Murat, and N.~Remli.
\newblock Linear programming with interval right hand sides.
\newblock {\em Int. Trans. Oper. Res.}, 17(3):397--408, 2010.

\bibitem{GarHla2019c}
E.~Garajov\'{a}, M.~Hlad\'{\i}k, and M.~Rada.
\newblock Interval linear programming under transformations: optimal solutions
  and optimal value range.
\newblock {\em Cent. Eur. J. Oper. Res.}, 27(3):601--614, September 2019.

\bibitem{GarRad2023}
E.~Garajov\'{a} and M.~Rada.
\newblock Interval transportation problem: feasibility, optimality and the
  worst optimal value.
\newblock {\em Cent. Eur. J. Oper. Res.}, 31(3):769--790, September 2023.

\bibitem{GarRad2024}
E.~Garajov\'{a} and M.~Rada.
\newblock A quasi-extreme reduction for interval transportation problems.
\newblock In H.~Moosaei et~al., editors, {\em Dynamics of Information Systems.
  DIS 2023}, volume 14321 of {\em LNCS}, pages 83--92, Cham, 2024. Springer.

\bibitem{Ger1981}
W.~Gerlach.
\newblock {Zur L\"osung linearer Ungleichungssysteme bei St\"orung der rechten
  Seite und der Koeffizientenmatrix}.
\newblock {\em Math. Operationsforsch. Stat., Ser. Optimization}, 12:41--43,
  1981.
\newblock in German.

\bibitem{GhoDeb2020}
D.~Ghosh, A.~K. Debnath, and W.~Pedrycz.
\newblock A variable and a fixed ordering of intervals and their application in
  optimization with interval-valued functions.
\newblock {\em Int. J. Approx. Reason.}, 121:187--205, June 2020.

\bibitem{HanWal2004}
E.~R. Hansen and G.~W. Walster.
\newblock {\em Global Optimization Using Interval Analysis}.
\newblock Marcel Dekker, New York, 2nd edition, 2004.

\bibitem{Hla2009b}
M.~Hlad\'{\i}k.
\newblock Optimal value range in interval linear programming.
\newblock {\em Fuzzy Optim. Decis. Mak.}, 8(3):283--294, 2009.

\bibitem{Hla2011b}
M.~Hlad\'{\i}k.
\newblock Optimal value bounds in nonlinear programming with interval data.
\newblock {\em TOP}, 19(1):93--106, 2011.

\bibitem{Hla2013b}
M.~Hlad\'{\i}k.
\newblock Weak and strong solvability of interval linear systems of equations
  and inequalities.
\newblock {\em Linear Algebra Appl.}, 438(11):4156--4165, 2013.

\bibitem{Hla2014a}
M.~Hlad\'{\i}k.
\newblock How to determine basis stability in interval linear programming.
\newblock {\em Optim. Lett.}, 8(1):375--389, 2014.

\bibitem{Hla2017c}
M.~Hlad\'{\i}k.
\newblock Interval convex quadratic programming problems in a general form.
\newblock {\em Cent. Eur. J. Oper. Res.}, 25(3):725--737, 2017.

\bibitem{Hla2024a}
M.~Hlad\'{\i}k.
\newblock Strong solvability of restricted interval systems and its
  applications in quadratic and geometric programming.
\newblock {\em Linear Algebra Appl.}, 693:4--21, July 2024.

\bibitem{Hla2025a}
M.~Hlad\'{\i}k.
\newblock {\em Interval Linear Programming and Extensions}.
\newblock Springer, Cham, 2025.

\bibitem{HlaHar2023au}
M.~Hlad\'{\i}k and D.~Hartman.
\newblock Absolute value linear programming.
\newblock preprint arXiv: 2307.03510, 2023.

\bibitem{HlaMoo2024au}
M.~Hlad\'{\i}k, H.~Moosaei, F.~Hashemi, S.~Ketabchi, and P.~M. Pardalos.
\newblock An overview of absolute value equations: {From} theory to solution
  methods and challenges.
\newblock {\em Comput. Optim. Appl.}, 2025.
\newblock in press.

\bibitem{Jan1988}
C.~Jansson.
\newblock A self-validating method for solving linear programming problems with
  interval input data.
\newblock In U.~Kulisch and H.~J. Stetter, editors, {\em Scientific computation
  with automatic result verification}, volume~6 of {\em Computing Suppl.},
  pages 33--45, Wien, 1988. Springer.

\bibitem{Kon2001}
J.~Kon\'{\i}\v{c}kov\'a.
\newblock Sufficient condition of basis stability of an interval linear
  programming problem.
\newblock {\em ZAMM, Z. Angew. Math. Mech.}, 81(Suppl. 3):677--678, 2001.

\bibitem{Kra1975}
R.~Krawczyk.
\newblock {Fehlerabsch{\"{a}}tzung bei linearer Optimierung}.
\newblock In K.~Nickel, editor, {\em Interval Mathemantics: Proceedings of the
  International Symposium, Karlsruhe, West Germany, May 20-24, 1975}, volume~29
  of {\em LNCS}, pages 215--222. Springer, 1975.
\newblock in German.

\bibitem{KumDee2024}
S.~Kumar, Deepmala, M.~Hlad\'{\i}k, and H.~Moosaei.
\newblock Characterization of unique solvability of absolute value equations:
  {An} overview, extensions, and future directions.
\newblock {\em Optim. Lett.}, 18(4):889--907, May 2024.

\bibitem{LiXia2015}
W.~Li, M.~Xia, and H.~Li.
\newblock New method for computing the upper bound of optimal value in interval
  quadratic program.
\newblock {\em J. Comput. Appl. Math.}, 288(0):70--80, 2015.

\bibitem{Liu2008}
S.-T. Liu.
\newblock Posynomial geometric programming with interval exponents and
  coefficients.
\newblock {\em Eur. J. Oper. Res.}, 186(1):17--27, 2008.

\bibitem{Man2007}
O.~L. Mangasarian.
\newblock Absolute value programming.
\newblock {\em Comput. Optim. Appl.}, 36(1):43--53, 2007.

\bibitem{Man2015b}
O.~L. Mangasarian.
\newblock Unsupervised classification via convex absolute value inequalities.
\newblock {\em Optim.}, 64(1):81--86, 2015.

\bibitem{ManMey2006}
O.~L. Mangasarian and R.~R. Meyer.
\newblock Absolute value equations.
\newblock {\em Linear Algebra Appl.}, 419(2):359--367, 2006.

\bibitem{Mez2020}
F.~Mezzadri.
\newblock On the solution of general absolute value equations.
\newblock {\em Appl. Math. Lett.}, 107:106462, 2020.

\bibitem{MooKea2009}
R.~E. Moore, R.~B. Kearfott, and M.~J. Cloud.
\newblock {\em Introduction to Interval Analysis}.
\newblock SIAM, Philadelphia, PA, 2009.

\bibitem{Mra1998}
F.~Mr\'az.
\newblock Calculating the exact bounds of optimal values in {LP} with interval
  coefficients.
\newblock {\em Ann. Oper. Res.}, 81:51--62, 1998.

\bibitem{OetPra1964}
W.~Oettli and W.~Prager.
\newblock Compatibility of approximate solution of linear equations with given
  error bounds for coefficients and right-hand sides.
\newblock {\em Numer. Math.}, 6:405--409, 1964.

\bibitem{PolRoh1993}
S.~Poljak and J.~Rohn.
\newblock Checking robust nonsingularity is {NP}-hard.
\newblock {\em Math. Control Signals Syst.}, 6(1):1--9, 1993.

\bibitem{RexRoh1998}
G.~Rex and J.~Rohn.
\newblock Sufficient conditions for regularity and singularity of interval
  matrices.
\newblock {\em SIAM J. Matrix Anal. Appl.}, 20(2):437--445, 1998.

\bibitem{Roh1980}
J.~Rohn.
\newblock Duality in interval linear programming.
\newblock In K.~Nickel, editor, {\em Interval Mathematics, Proc. Int. Symp.,
  Freiburg, 1980}, pages 521--529, New York, 1980. Academic Press.

\bibitem{Roh2006b}
J.~Rohn.
\newblock Interval linear programming.
\newblock In M.~Fiedler et~al., editors, {\em Linear Optimization Problems with
  Inexact Data}, chapter~3, pages 79--100. Springer, New York, 2006.

\bibitem{Roh2006a}
J.~Rohn.
\newblock Solvability of systems of interval linear equations and inequalities.
\newblock In M.~Fiedler et~al., editors, {\em Linear Optimization Problems with
  Inexact Data}, chapter~2, pages 35--77. Springer, New York, 2006.

\bibitem{Roh2009}
J.~Rohn.
\newblock Forty necessary and sufficient conditions for regularity of interval
  matrices: {A} survey.
\newblock {\em Electron. J. Linear Algebra}, 18:500--512, 2009.

\bibitem{RohHoo2014}
J.~Rohn, V.~Hooshyarbakhsh, and R.~Farhadsefat.
\newblock An iterative method for solving absolute value equations and
  sufficient conditions for unique solvability.
\newblock {\em Optim. Lett.}, 8(1):35--44, 2014.

\bibitem{WuShe2021}
S.~Wu and S.~Shen.
\newblock On the unique solution of the generalized absolute value equation.
\newblock {\em Optim. Lett.}, 15:2017--2024, 2021.

\bibitem{YamFuk2014}
S.~Yamanaka and M.~Fukushima.
\newblock A branch-and-bound method for absolute value programs.
\newblock {\em Optim.}, 63(2):305--319, 2014.

\bibitem{ZamHla2021a}
M.~Zamani and M.~Hlad\'{\i}k.
\newblock A new concave minimization algorithm for the absolute value equation
  solution.
\newblock {\em Optim. Lett.}, 15(6):2241--2254, September 2021.

\end{thebibliography}

\end{document}